\documentclass[a4paper, 11pt,reqno]{amsart}

\usepackage{amsmath,amsthm,amsbsy,amssymb}
\usepackage{arydshln} 
\usepackage{booktabs} 

\makeatletter
\newcommand{\leqnos}{\tagsleft@true\let\veqno\@@leqno}
\newcommand{\reqnos}{\tagsleft@false\let\veqno\@@eqno}
\reqnos
\makeatother

\usepackage{xcolor} 
\definecolor{orange}{rgb}{1,0.5,0}
\definecolor{Red}{rgb}{.795,0.015,0.017}
\definecolor{Ggreen}{rgb}{0.,0.675,0.0128}
\definecolor{Bblue}{rgb}{0.16,.32,0.91}
 \usepackage[backref=page]{hyperref}
\hypersetup{
    colorlinks = true,
linkcolor={red},
urlcolor={blue},
citecolor={Ggreen},    
urlcolor = {blue},
citebordercolor = {0.33 .58 0.33},
 linkbordercolor = {0.99 .28 0.23},
 pdftitle={Statistical Distribution of Fermat Quotients},
 pdfauthor={ACVZ}
}

\usepackage{mathrsfs}

\def\scB{\mathscr{B}}

\def\cB{\mathcal B}

\def\cD{\mathcal D}

\newcommand{\cG}{\mathcal G}

\newcommand{\cV}{\mathcal V}

\newcommand{\cT}{\mathcal T}

\def\cP{\mathcal P}

\def\cS{\mathcal S}
\def\cX{\mathcal X}

\def\fq{\mathfrak q}

\newcommand{\FQM}[1]{{\mathtt{FQM}}(#1)}
\newcommand{\FQMp}{{\mathtt{FQM}}(p)}

\usepackage{bm}
\newcommand*{\B}[1]{\ifmmode\bm{#1}\else\textbf{#1}\fi}
\newcommand{\bh}{\B{h}}
\newcommand{\bv}{\B{v}}

\newcommand{\bx}{\B{x}}

\def\FF{\mathbb{F}}

\def\ZZ{\mathbb{Z}}

\newcommand{\sdfrac}[2]{\mbox{\small$\displaystyle\frac{#1}{#2}$}}

\newcommand{\tdfrac}[2]{\mbox{\tiny$\displaystyle\frac{#1}{#2}$}}

\usepackage{xspace}

\newcommand*\dd{\mathop{}\!\mathrm{d}}
\DeclareMathOperator*{\Discrepancy}{D}
\DeclareMathOperator*{\DiscrepancyN}{\Discrepancy{\hspace{-3pt}}_\mathit{N}}

\renewcommand{\pmod}[1]{\left( \mathrm{ mod\;}#1\right)}

\theoremstyle{plain}
\newtheorem{theorem}{Theorem}
\newtheorem{corollary}{Corollary}

\newtheorem{lemma}{Lemma}[section]
\newtheorem{proposition}{Proposition}[section]

\newtheorem{remark}{Remark}[section]
\newtheorem*{remark*}{Remark}
\theoremstyle{remark}

\theoremstyle{definition}

\usepackage{subfigure}  
\usepackage{float}
\usepackage[pdftex]{graphicx}
\graphicspath{{IMAGES/}{IMAGES1/}{../IMAGES1/}{./}}
\usepackage[]{caption}
\usepackage{cite}

\usepackage{pgfplots}
\pgfplotsset{compat=1.8}
\usetikzlibrary{spy}  

\date{}
\begin{document}

\title[Pattern formation Statistics on Fermat Quotients]
{Pattern formation Statistics on Fermat Quotients}

\author[C. Cobeli, A. Zaharescu, Z. Zhang]{Cristian Cobeli, Alexandru Zaharescu, Zhuo Zhang}

\address[CC, AZ]{
``Simion Stoilow'' Institute of Mathematics of the Romanian Academy,~21 Calea Griviței Street, 
P. O. Box 1-764, Bucharest 014700, Romania}

\address[AZ, ZZ]{
Department of Mathematics,
University of Illinois at Urbana-Champaign,
Altgeld Hall, 1409 W. Green Street,
Urbana, IL, 61801, USA\vspace{7pt}}

\email{cristian.cobeli@imar.ro}  
\email{zaharesc@illinois.edu}  
\email{zhuoz4@illinois.edu}

\makeatletter
\@namedef{subjclassname@2020}{%
  \textup{2020} Mathematics Subject Classification}
\makeatother


\begin{abstract}
Despite their simple definition as 
$\fq_p(b):=\frac{b^{p-1}-1}{p} \pmod p$, for \mbox{$0\le b \le p^2-1$} and $\gcd(b,p)=1$,
and their regular arrangement in a $p\times(p-1)$ 
Fermat quotient matrix $\FQM{p}$ of integers from $[0,p-1]$,
Fermat quotients modulo $p$ are well known for their overall lack of \mbox{regularity}.
Here, we discuss this contrasting effect by proving that, on the one hand,
any line of the matrix behaves like an analogue of a randomly distributed sequence of numbers, 
and on the other hand, the spatial statistics of distances on regular $N$-patterns  
confirm the natural expectations.
\end{abstract}
\maketitle

\section{Introduction}

The widely known arithmetical fact brought out in front by Fermat's little theorem  has in the background
a deep wealth of arithmetic, algebraic, geometric, or analytic phenomena.
A particular branch of this subject, which has been recently extensively studied,
develops the theory of $\delta_p$-differentiation and 
uses it, among other things, to bound the number of rational points on curves
(see Joyal~\cite{Joy1985}, Buium~\cite{Bui1996}, Jeffries~\cite{Jef2023} and the references therein).

For the first time, the \textit{Fermat quotients} defined as the ratios $Q_p(n)=(n^{p-1}-1)/p$, 
which are integers, proved to be important for recognizing the role of the prime numbers
for which $(n^{p-1}-1)/p$ is again divisible by $p$ in the study of Fermat's Last Theorem
(see Mirimanoff~\cite{ Mir1895, Mir1911}, Lerch~\cite{Ler1905}, 
Wieferich~\cite{Wie1909}, Shanks and Williams~\cite{SW1981}, Suzuki~\cite{Suz1994}).
Fermat quotients have been studied over the years for their remarkable arithmetic and algebraic properties
(see  Eisenstein~\cite{Eis1850}, 
Johnson~\cite{Joh1977,Joh1978}, 
Ernvall and Mets{\"a}nkyl{\"a}~\cite{EM1997}, 
Sauerberg and Shu~\cite{SS1997},
Agoh~\cite{Ago2008}),
particularly for their relevance to number-theoretic 
problems
(see Shparlinski~\cite{Shp2011b,Shp2012,Shp2013,Shp2011a} and
Alexandru et al.~\cite{ACVZ2023})
and their
association to pseudo-random number generators, prime factorization and 
cryptographic algorithms
(see Lehmer~\cite{Leh1981}, 
Lenstra~\cite{Len1979}, 
Ostafe and Shparlinski~\cite{OS2011}, 
Shparlinski~\cite{Shp2011c}, 
Chang~\cite{Cha2012}, 
Chen et al.~\cite{COW2010, CW2014, WXCK2019,CW2012},
Liu and Liu~\cite{LL2022,LL2023}.

Working with the standard representatives modulo $p$ from
$\{0,1,\dots,p-1\}$, the Fermat quotients $Q_p(n)\pmod p$ 
of \textit{base} $n$ and \textit{exponent} $p$,
for $1\le n<p^2$ and $p\nmid n$,
denoted
\begin{equation}\label{eqfqa}
   \fq_p(b):=\frac{b^{p-1}-1}{p} \pmod p, \text{ for $1\le b \le p^2-1$ and $\gcd(b,p)=1$},
\end{equation}
are arranged in order in the $\FQM{p}$ matrix, 
$p-1$ per line, for a total of $p$ successive lines
(see~\cite{EM1997,ACVZ2022} and an example in Table~\ref{Table1}).

Although the elements of $\FQMp$ are not independent from each other 
(in Proposition~\ref{PropositionFQ} there are stated some relations between them), 
numerical calculations indicate that there is still a pronounced 
aspect of pseudo-randomness in the Fermat quotient matrix. 
In this article our main goal is to substantiate this observation, 
generalizing a result~\cite[Theorem 3]{ACVZ2022} according to which,
beyond the inherent symmetries, if $p$ is sufficiently large,
two elements of the matrix, regardless of their position,
are not found in any particular relationship in terms of size. 

But first we will prove a common but somewhat peculiar characteristic of 
the Fermat quotient matrix.
We call the first line of $\FQM{p}$ a \textit{Fermat quotient point}. 
Then the next theorem shows that the average distance of the Fermat quotient point
to a straight line is, in the limit, equal to $1/3$.
Note that in the mod $p$ modular context, straight lines are composed of 
equally spaced  parallel segments, and their number increases as the slope of the line 
increases (see~Figure~\ref{FigureCurves12} for some typical examples of such lines).

\begin{theorem}\label{Theorem1}
Let $\{C_p\}_{p\in\cP}$ and $\{D_p\}_{p\in\cP}$
be two sequences of integers indexed by prime numbers.
Suppose $C_p\neq 0$ and $C_p=o(p^{1/12}\log^{-2/3} p)$ 
as $p$ tends to infinity.
Then
\begin{equation}\label{eqTheorem1}
    \sdfrac{1}{p}\sum_{b=1}^{p-1}
    \Big|\sdfrac{\fq_p(b)}{p}-\Big\{C_p\sdfrac{b}{p}+D_p\Big\}\Big|
    = \sdfrac{1}{3}+O\Big(C_p^{3/5}p^{-1/20}\log^{2/5}p\Big)\,, 
\end{equation}     
where $\big\{C_pb/p+D_p\big\}$ denotes the fractional part of 
$C_pb/p+D_p$.
\end{theorem}
In Theorem~\ref{Theorem1}, we allowed the slopes of the lines to change as $p$ increases, 
with the result remaining valid even when the slopes increase
together with the number of the segments that compose the lines.

The common version, with the distance between the Fermat quotient point
and any fixed line, is the object of the following corollary.
\begin{corollary}\label{Corollary1}
  Let $C\neq 0$ and $D$ be fixed.
Then, 
\begin{equation*}
  \sum_{b=1}^{p-1}\Big|
  \sdfrac{\fq_p(b)}{p}-\Big\{C\sdfrac{b}{p}+D\Big\}  \Big|
  =\sdfrac{p}{3}
+O_C\Big(p^{19/20}\log^{2/5}p\Big)\,.
\end{equation*}    
\end{corollary}

\smallskip
Our second subject refers to the entire $\FQMp$ matrix, with our objective 
being the comparison of the sizes of Fermat quotients in specific 
pre-established geometric positions.
In order to use the common notation for the entries in the matrix, we denote the elements
of $\FQM{p}$ by $A_{a,b}$, where $0\le a\le p-1$ and $1\le b\le p-1$, so that 
\begin{equation*}
    A_{a,b} := \fq_p(ap+b).
\end{equation*}
Next, let $N\ge 1$ and consider as fixed $N$ two-dimensional vectors $\bv_1,\dots,\bv_N\in \ZZ^2$, whose components will be denoted by $\bv_j=(s_j,t_j)$.
In the following, we will assume that all integers  $t_j$ are distinct pairwise, because otherwise,
the specific mutual relations among the vectors $\bv_1,\dots,\bv_N$ make the general result
we will prove to no longer be valid. 

We will use the set of vectors $\bv_1,\dots,\bv_N\in \ZZ^2$ 
as a spanning pattern that we apply across $\FQM{p}$ 
to compare among each other the size of Fermat quotients.
Let $\cV\in(\ZZ^2)^N$ denote the ordered set of the vectors that defines the pattern.

If $p$ becomes sufficiently large, applying the $\cV$-span to the elements of 
$\FQMp$ most of the time we will remain in the matrix, since $\cV$ is fixed. The 
only elements that are moved outside of $\FQMp$ by $\cV$ are near the border, and 
their total number is at most $O(p)$.
In the following we denote by $\cG_N(p)\subset\ZZ^2$ the set of places
that are kept inside while translated by~$\cV$, that is,
\begin{equation}\label{eqdefG}
    \cG_N(p) := \bigcap_{j=1}^N
    \Big\{(a,b)\in [0,p-1]\times [1,p-1]
    : 0\le a+s_j\le p-1,\ 
    1\le b+t_j\le p-1
    \Big\}\,.
\end{equation}
Let us note that the collection of pairs of indices from 
$ \cG_N(p)$ cut out from $\FQMp$ parts that together are close to a square table  
of 
\begin{equation}\label{eqCardGN}
    \#\cG_N(p) = p^2+O_M(p)
\end{equation}
points, where $M=\max\limits_{1\le j\le N}\|\bv_j\|_{\infty}$.

Finally, for each $(a,b)\in\cG_N(p)$, consider the normalized 
$\cV$-spanned vectors of Fermat quotients mod $p$ defined by
\begin{equation}\label{eqdefx}
    \bx_{a,b}(p):=\frac 1p 
\big(A_{(a,b)+\bv_1},\dots,A_{(a,b)+\bv_N}\big)\in[0,1]^N\,.
\end{equation}

We will show that if $p$ is sufficiently large, then the set
\begin{equation}\label{eqX}
    \cX(p) := \big\{\bx_{a,b}(p) : (a,b)\in\cG(p)\big\} \subset
    [0,1]^N
\end{equation}
is uniformly distributed.

For any permutation $\sigma\in S_N$, we denote by
$T(\sigma,N)$ the $N$-dimensional \textit{polyhedron} 
whose coordinates permuted by $\sigma$ are ordered increasingly, that is,
\begin{equation}\label{eqThtrahedron}
    T(\sigma,N):=\big\{\bx\in [0,1]^N : \bx = (x_1,\ldots,x_N),\  
    x_{\sigma(1)}<x_{\sigma(2)}<\cdots<x_{\sigma(N)}\big\}.
\end{equation}

Next, by employing $T(\sigma,N)$ and the uniform distribution of $\cX(p)$, 
we can precisely estimate the size of clusters for all possible order relations among the elements of $\FQMp$.

\begin{theorem}\label{Theorem2}
Let $N\ge 1$ and let $\bv_1,\dots,\bv_N\in\ZZ^2$ be fixed
and having the property that their second components
are different from each other.
Suppose $p$ is prime and
$\cG_N(p)$,  $\bx_{a,b}(p)$ are defined by~\eqref{eqdefG}
and~\eqref{eqdefx}.
Then, for any permutation $\sigma\in S_N$, we have
\begin{equation}\label{eqTH}
    \#\big\{(a,b)\in\cG_N(p):
       \bx_{a,b}(p)\in T(\sigma,N)\big\}
       =\frac{p^2}{N!}+
       O_{M,N}\left(p^{(2N+1)/(N+1)}\log^{N/(N+1)}p\right),
\end{equation}
where   $M:=\max\limits_{1\le j\le N}\|\bv_j\|_{\infty}$.
\end{theorem}
We remark that the size of the main term in~\eqref{eqdefx} is proportional to the volume of
the $N$-dimensional polyhedron set put in Lemma~\ref{LemmaTetrahedron}.

Theorem~\ref{Theorem2} generalizes the specific case of the pairs~\cite{ACVZ2022}, 
where the number of two generic elements in the matrix $\FQMp$ that are relative to each other 
in a certain a priori fixed geometric position. Thus, if $N=2$, the result proves 
that in roughly half of the cases, one element is larger than the other, and vice versa.
\begin{corollary}[\!\!{\cite[Theorem 3]{ACVZ2022}}]\label{Corollary2}
    Let $s,t$ be integers not both equal to zero
    and let 
\begin{equation*}
  \cG(p)=\{(a,b)\in[0,p)\times[1,p)
    : 0\le a+s<p,\ 1\le b+t<p\}\cap\ZZ^2.    
\end{equation*}
Then,
\begin{equation*}
    \#\big\{(a,b)\in\cG(p):
     \fq_{a,b}\lessgtr \fq_{a+s,b+t}\big\}
       =\frac{p^2}{2}+O_{|s|,|t|}\big(p^{5/3}\log^{2/3}p\big),   
\end{equation*}
where $\lessgtr$ stands for either $\le$ or $\ge$.
\end{corollary}

The condition in Theorem~\ref{Theorem2} that all the second components of $\bv_j$'s
must be distinct is necessary and guarantees the validity of the result for nearly 
all possible displacement vectors, when $M$ becomes sufficiently large.
We remark that even though the result may still hold true even if some of 
the second components of the displacement vectors are equal to some other,
the regime changes in these cases.
Indeed, if fewer and fewer second components are distinct, 
regardless of the validity of the theorem's result, 
the set of pairs whose cardinality is estimated by~\eqref{eqTH} 
becomes more and more regular, while the set of vectors 
$\cX(p)\cap T(\sigma,N)$ may 
no longer be equidistributed in the cube $[0,1]^N$.
In the Addenda, we show in Figures~\ref{FigureXT1}--\ref{FigurePolyhedrons3d} 
the graphical representations of several such typical situations.

\section{Preparatory Lemmas}\label{SectionLemmas}
\begin{center}
\begin{table}[tb]
 \centering
 \captionsetup{width=.85\linewidth}
  \caption{\normalfont The Fermat quotient matrix $\FQM{11}$
 with entries $\fq_{11}(n)=\tdfrac{n^{10}-1}{11}\pmod {11}$.}
 \vspace{-5pt}
 \setlength{\tabcolsep}{9pt}
\renewcommand{\arraystretch}{1.21}
 \footnotesize
\begin{tabular}{@{}l cccccccccc @{}}
\toprule 
$b$ 	 & 1 & 2 & 3 & 4 & 5 & 6 & 7 & 8 & 9 & 10 
\\ 
$b^{-1}$ & 1  & 6  & 4  & 3  & 9  & 2  & 8  & 7  & 5  & 10 
\\ \midrule
$\fq_p(b)$ & 
0  & 5  & 0  & 10  & 7  & 5  & 2  & 4  & 0  & 1 \\
$\fq_p(p+b)$ & 
10  & 10  & 7  & 7  & 9  & 3  & 5  & 8  & 6  & 2 \\
$\fq_p(2p+b)$ &  
9  & 4  & 3  & 4  & 0  & 1  & 8  & 1  & 1  & 3 \\
$\fq_p(3p+b)$ & 
8  & 9  & 10  & 1  & 2  & 10  & 0  & 5  & 7  & 4 \\
$\fq_p(4p+b)$ &
7  & 3  & 6  & 9  & 4  & 8  & 3  & 9  & 2  & 5 \\
$\fq_p(5p+b)$ &
6  & 8  & 2  & 6  & 6  & 6  & 6  & 2  & 8  & 6 \\
$\fq_p(6p+b)$ &
5  & 2  & 9  & 3  & 8  & 4  & 9  & 6  & 3  & 7 \\
$\fq_p(7p+b)$ &
4  & 7  & 5  & 0  & 10  & 2  & 1  & 10  & 9  & 8 \\
$\fq_p(8p+b)$ &
3  & 1  & 1  & 8  & 1  & 0  & 4  & 3  & 4  & 9 \\
$\fq_p(9p+b)$ &
2  & 6  & 8  & 5  & 3  & 9  & 7  & 7  & 10  & 10 \\
$\fq_p(10p+b)$ &
1  & 0  & 4  & 2  & 5  & 7  & 10  & 0  & 5  & 0\\
    \bottomrule
\end{tabular} \label{Table1}
\end{table}
\end{center}

The elements of the Fermat quotient matrix are related to each
other through various basic relations, including some that are indicated in the next proposition (see also~\cite{EM1997}).
\begin{proposition}\label{PropositionFQ}
  Let $A_{a,b}$ be the elements of $\FQMp$, for $0\le a<p$
and $1\le b<p$. Then, with any indices and entries taken modulo $p$, we have
\begin{enumerate}\setlength\itemsep{3pt}
\renewcommand{\theenumi}{{\upshape\roman{enumi}}}
    \item\label{partP31}  $A_{a,b} = \fq_p(ap+b)$;
    \item\label{partP32}  $A_{a,b} = A_{p-1-a,p-b}$;
    \item\label{partP33}  $A_{a,b} = \fq_p(b)-ab^{-1}$;
    \setlength\itemsep{.1pt}
    \item\label{partP34}  $A_{a+s,b} = A_{a,b}-sb^{-1}$;
    \item \label{partP35} 
$e_p\big(A_{a,b_1b_2}\big) = e_p\big(A_{a,b_1}\big)
e_p\big(A_{a,b_1}\big)
e_p\big(a(b_1+b_2-1)b_1^{-1}b_2^{-1}\big)$,
where $e_p(x):=e^{\frac{2\pi i x}{p}}$.
\end{enumerate}
\end{proposition}
\begin{proof}
    Part~\eqref{partP31} is listed as the definition and the others are straightforward 
calculations using 
$(ap+b)^{p-1}\equiv b^{p-1}-apb^{p-2}\pmod{p^2}$.
Then, via the translation of the notation
$\fq_p(n)=\sdfrac{n^{p-1}-1}{p}\pmod p$, this particularly implies 
$\fq_p(ap+b)\equiv \fq_p(b)-ab^{-1}\pmod p$.
Also, note that $A_{a+s,b}$ is equal to
\begin{equation*}
  \begin{split}
   \fq_p\big({(a+s)}p+b\big)
  &=\frac{\big((a+s)p+b\big)^{p-1}-1}{p}\pmod p  \\
  &=\frac{(p-1)(a+s)pb^{p-2}+b^{p-1}-1}{p}\pmod p,
  \end{split}
\end{equation*}
which further leads to
\begin{equation*}
  \begin{split}
  A_{a+s,b} = -(a+s)b^{-1}+\frac{b^{p-1}-1}{p}\pmod p
  = -(a+s)b^{-1}+\fq_p(b).
  \end{split}
\end{equation*}

The proof of part~\eqref{partP35} yields as follows. Using part~\eqref{partP33}, we have:
\begin{equation}\label{eqMult1}
    e_p\big(A_{a,b_1b_2}\big)
    = e_p\big(\fq_p(b_1b_2)-ab_1^{-1}b_2^{-1}\big)
    = e_p\big(\fq_p(b_1)\big)e_p\big(\fq_p(b_2)\big)e_p\big(-ab_1^{-1}b_2^{-1}\big),
\end{equation}
because $\fq_p(b_1b_2)=\fq_p(b_1)b_2^{p-1}+\fq_p(b_2)$, which implies,
$\fq_p(b_1b_2)\equiv\fq_p(b_1)+\fq_p(b_2)\pmod p$
for $\gcd(b_1b_2,p)=1$. 
Then, the sequence of equalities in~\eqref{eqMult1} is further continued by
\begin{equation*}
  \begin{split}
   e_p\big(A_{a,b_1b_2}\big)
   &= e_p\big(\fq_p(b_1)-ab_1^{-1}\big)e_p\big(\fq_p(b_2)-ab_2^{-1}\big)
    e_p\big(ab_1^{-1}+ab_2^{-1}-ab_1^{-1}b_2^{-1}\big)\\
    &= e_p\big(A_{a,b_1}\big)e_p\big(A_{a,b_2}\big)
    e_p\big(a(b_1+b_2-1)b_1^{-1}b_2^{-1}\big).      
  \end{split}
\end{equation*}
\end{proof}

Let us remark that parts~\eqref{partP33} and~\eqref{partP34} of Proposition~\ref{PropositionFQ}
unfold a link between the Fermat quotient point and 
the other entries in the matrix $\FQMp$,
indicating that the elements of the Fermat quotient matrix 
might share similarities with the type of distribution and the expected 
spread of geometric patterns of the inverses modulo $p$ 
(see~\cite{CGZ2003, CVZ2003, CVZ2000}).
\begin{figure}[htb]
 \centering
 \hfill 
    \includegraphics[width=0.48\textwidth]{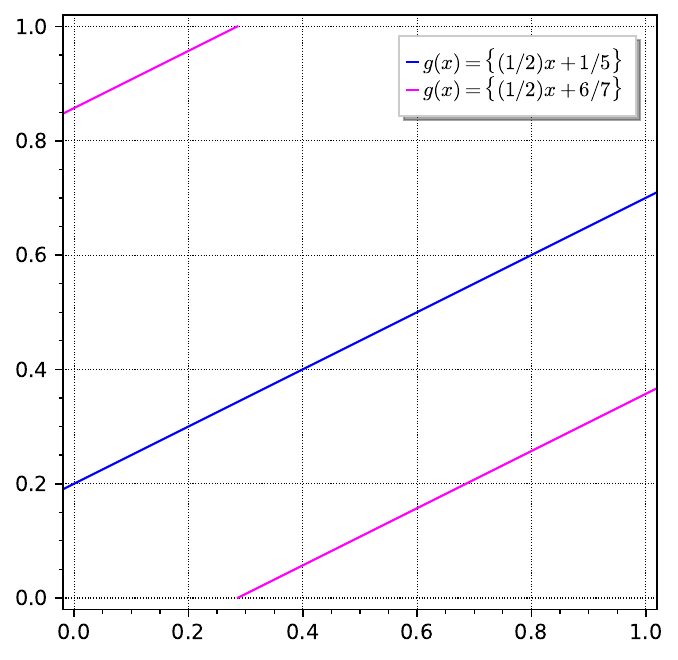}\ \ 
    \includegraphics[width=0.48\textwidth]{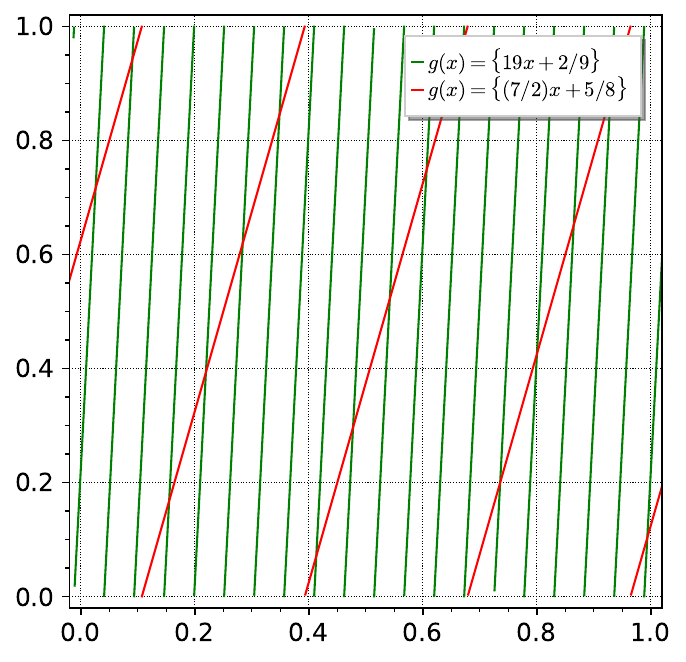}
\hfill\mbox{}
 \label{FigPage3} \label{FigCurves12}
\caption{The graphs of four straight lines mod $1$ with four different slopes and $y$-intercepts.
} 
 \label{FigureCurves12}
 \end{figure}

\begin{lemma}\label{LemmaIntegral2}
    Let $C, D>0$ be real numbers and let $g(x):=\{Cx+D\}$. Then
    \begin{equation*}\label{eqLemmaIntegral2}
  I(g) := \int_0^1\int_0^1 |y-g(x)|\dd{x}\dd{y} 
  = \sdfrac{1}{3}\,.   
\end{equation*}
\end{lemma}
\begin{proof}
Let us note that the integrand function is piecewise continuous and bounded, 
thus Fubini's theorem can be applied. Then, since $0\le g(x)<1$, we have
\begin{equation*}
\begin{split}
   \int_0^1 |y-g(x)|\dd{y} & =     \int_0^{g(x)}(g(x)-y)\dd{y}+
    \int_{g(x)}^1(y-g(x))\dd{y}\\
    & = \big(g(x)\big)^2 -\left.\frac{y^2}{2}\right\rvert_0^{g(x)} 
    + \left.\frac{y^2}{2}\right\rvert_{g(x)}^1 
    -g(x)\big(1-g(x)\big)\,,
\end{split}
\end{equation*}
and, as a result,
\begin{equation}\label{eqIf}
     I(g) 
     = \int_0^1\left(\big(g(x)\big)^2 -g(x)+\sdfrac{1}{2}\right)\dd{x}.
\end{equation}
\begin{remark}\label{RemarkPeriod}
We can simplify the calculation of the integral noticing that the function
$x\mapsto \{Cx+D\}$ is periodic of period $1/C$. Indeed, we have
\begin{equation*}
       \Big\{C\Big(x+\sdfrac{1}{C}\Big)+D\Big\} = \big\{Cx+1+D\big\} 
       = \{Cx+D\}\,.
\end{equation*}
\end{remark}
Next, one checks that if $x\in [0,1/C)$, then
\begin{equation}\label{eqAxB}
    \{Cx+D\} = 
    \begin{cases}
        Cx+D & \text{ if $0\le x \le \sdfrac{1-D}{C}$,}\\[6pt]
        Cx+D-1 & \text{ if $\sdfrac{1-D}{C}\le x \le \sdfrac{1}{C}$}.
    \end{cases}
\end{equation}
Then, on combining~\eqref{eqIf} and~\eqref{eqAxB}, we have
\begin{equation*}
  \begin{split}
       I\big(g\big) 
       &= C\int_0^{1/C}\left(\{Cx+D\}^2 - \{Cx+D\}+\sdfrac{1}{2}\right)\dd{x}\\
       & = C \int_0^{\frac{1-D}{C}}
       \left((Cx+D)^2 - (Cx+D)+\sdfrac{1}{2}\right)\dd{x}\\
       &\phantom{ (Cx+D)^2 -  (Cx+D)}
       +  C \int_{\frac{1-D}{C}}^{\frac{1}{C}}
         \left((Cx+D-1)^2 - (Cx+D-1)+\sdfrac{1}{2}\right)\dd{x} .
  \end{split}
\end{equation*}
Summing the terms alike in the two integrals, yields
\begin{equation*}
  \begin{split}
       I\big(g\big) 
       &= C\int_0^{1/C}\left((Cx+D)^2 - (Cx+D)+\sdfrac{1}{2}\right)\dd{x}\\
       &\phantom{ (Cx+D)^2 -  (Cx+D)}
       +  C \int_{\frac{1-D}{C}}^{\frac{1}{C}}
         \big(- 2(Cx+D)+ 1 + 1\big)\dd{x} \\
        &= \left.C\left(\sdfrac{(Cx+D)^3}{3C} - \sdfrac{(Cx+D)^2}{2C} +\sdfrac{x}{2}\right) 
           \right\rvert_{0}^{1/C}
           + \left.C\left(\sdfrac{-2(Cx+D)^2}{2C} +2x\right) 
           \right\rvert_{(1-D)/C}^{1/C}.
  \end{split}
\end{equation*}
Next, a straightforward calculation leads to the conclusion that 
$I\big(g\big)=1/3$, which concludes the proof of the lemma.
\end{proof}

For any finite sequence $\cS\subset[0,1]$, and any $\alpha,\beta\in [0,1]$, 
let 
$\Discrepancy(\cS;\alpha,\beta):=\big|\cS\cap [\alpha,\beta]\big|-|\cS|(\beta-\alpha)$
be its \textit{discrepancy} in the interval
$[\alpha,\beta]$.
Then, the \textit{uniform discrepancy} of $\cS$, denoted by~$\Discrepancy(\cS)$, is defined by
\begin{equation*}
   \begin{split}
   \Discrepancy(\cS):=\frac{1}{|\cS|}\cdot
      \sup_{0\le\alpha\le\beta\le 1}\big|\Discrepancy(\cS;\alpha,\beta)\big|\,.
   \end{split}
\end{equation*}
The 
uniform discrepancy can be bounded by the Erd\H os-Tur\'an inequality.
\begin{lemma}[Erd\H os-Tur\'an~{\cite[Corollary 1.1, Chap. 1]{Mon1994}}]\label{LemmaETI}
    For any integer $K>1$    
\begin{equation}\label{eqETI}
    \big|\Discrepancy(\cS)\big| \le
      \frac{|\cS|}{K} +3\sum_{m=1}^K\frac{1}{m}
      \bigg|\sum_{s\in\cS}e(ms)\bigg|\,,
\end{equation} 
where $e(x):=\exp\left(2\pi i x\right)$.
\end{lemma}
We will need to estimate exponential sums with 
Fermat quotients, and to do so, we will make use of the following bound.
\begin{lemma}[Heath--Brown~{\cite{HB1996}}]\label{LemmaHB}
    For any integer $m$ relatively prime to $p$ we have
    \begin{equation}\label{eqHB}
        \sum_{\substack{X<n\le X+Y\\\gcd(n,p)=1}}
        e\left(\sdfrac{m\fq_p(n)}{p}\right)
        = O\big(Y^{1/2}p^{3/8}\big),
    \end{equation}
uniformly for $X,Y>1$.
\end{lemma}

The next lemma gives the volume of the hyper-polyhedron. 
\begin{lemma}\label{LemmaTetrahedron}
Let $\sigma$ be a permutation of size $N\ge 1$
and let $T(\sigma,N)\subset [0,1]^N$ be the polyhedron defined by
\begin{equation*}
    T(\sigma,N):=\big\{(x_1,\ldots,x_N)\in [0,1]^N : 
    x_{\sigma(1)}<x_{\sigma(2)}<\cdots<x_{\sigma(N)}\big\}.
\end{equation*}
Then the Lebesgue measure of $T(\sigma,N)$ is
\begin{equation*}
   \mu\big(T(\sigma,N)\big) = \frac{1}{N!}. 
\end{equation*}
\end{lemma}
\begin{proof}
  The result follows by noting the following recursion when we reduce the dimension:
\begin{equation*}
  \begin{split}
  \int_{0}^{1}\int_{0}^{z_N}
  \cdots
  \int_{0}^{z_3}\int_{0}^{z_2} 
  1 \dd z_1\dd z_2\cdots \dd z_{N}
  & = 
  \int_{0}^{1}\int_{0}^{z_N}      
  \cdots
  \int_{0}^{z_4}\int_{0}^{z_3} 
  z_2 \dd z_2\dd z_3\cdots \dd z_N\\
    & = 
  \int_{0}^{1}\int_{0}^{z_N}      
  \cdots
  \int_{0}^{z_5}\int_{0}^{z_4} 
  \sdfrac{z^2_3}{ 2} \dd z_3\dd z_4\cdots \dd z_N\\
    & = 
  \int_{0}^{1}\int_{0}^{z_N}      
  \cdots
  \int_{0}^{z_6}\int_{0}^{z_5} 
  \sdfrac{z^3_4}{ 2\cdot 3} \dd z_4\dd z_5\cdots \dd z_N\,.
  \end{split}
\end{equation*}

Differently, the same result is a consequence of the observation
that $\mu\big(T(\sigma,N)\big)$ is the same for any $\sigma$.
\end{proof}

We will also need a basic tool analogous to the 
Erd\H os-Tur\'an inequality, which is useful for classifying sequences in 
multi-dimensional spaces as uniformly distributed.
For this, consider the sequence $\bx=\{\bx_n\}_{n\ge 1}\subset[0,1)^N$ and denote
by $\scB$ a generic box in the unit cube:
\begin{equation*}
    \scB=\prod_{1\le j\le N}[a_j,b_j)\subset [0,1)^N.
\end{equation*}
Then the \textit{extreme discrepancy} of $\bx$ with respect to $\scB$ is defined by
\begin{equation*}
     \DiscrepancyN(\bx,R) :=
    \sup_{\scB\subset[0,1)^N}\bigg|
        \mu(\scB) - \sdfrac{1}{R}\#\big\{n\le R : \bv_n\in\scB\big\}
    \bigg|\,,
\end{equation*}
where $\mu$ denotes the Lebesgue measure.

The next lemma gives the Koksma--Sz\"usz inequality~\cite{Kok1950, Szu1952},
which bounds the extreme discrepancy in terms of the level of cancellation in the
exponential sums associated to the sequence in question.
It shows that if the sequence  $\bx=\{\bx_n\}_{n\ge 1}\subset[0,1)^N$
is uniformly distributed at random in $[0,1)^N$ and independently for any $n\ge 1$, 
then $\mu(\scB)$ approximates effectively the proportion
of the elements $\bx_n$ that belong to~$\scB$.

\begin{lemma}[Koksma--Sz\"usz]\label{LemmaETK}
For any integer $H>1$ and any sequence 
$\bx=\{\bx_n\}_{n\ge 1}\subset[0,1)^N$, we have
\begin{equation}\label{eqETK}
    \DiscrepancyN(\bx,R)\leq 
    C_N
    \left(\frac{2}{H+1}
    +\sum _{0<\|\bh\|_{{\infty}\leq H}}\frac  {1}{r(\bh)}
    \left|{\frac{1}{R}}\sum_{j=1}^{R}
    e\Big(2\pi i  \langle \bh,\bx_{j}\rangle\Big)\right|\right),
\end{equation}
where
$r(\bh)=\prod\limits_{i=1}^{N}\max\{1,|h_{i}|\}$
for
$\bh=(h_{1},\dots ,h_{N})\in \ZZ^{N}$, $\langle\cdot,\cdot\rangle$ 
denotes the inner product, and $C_N$ denotes an absolute constant depending on the dimension $N$. 
\end{lemma}

The following lemma estimates the exponential sum in~\eqref{eqETK}
for the sequence of elements in the Fermat quotient matrix translated by the 
displacements $\cV=\{\bv_1,\dots,\bv_N\}\subset\ZZ^2$.
\begin{lemma}\label{LemmaESk}
Let $p$ be prime, $H, M, N\ge 1$, and let
$\bv_j=(s_j,t_j)\in\ZZ^2$ for $1\le j\le N$ be fixed and chosen such that
$\max\limits_{1\le j\le N}\|\bv_j\|_\infty\le M$.
Assume that one of the integers $t_j$ is different from all the others,
and denote by $J$ its index.
Let 
\begin{equation*}
    \cG_N(p) := \bigcap_{j=1}^N
    \Big\{(a,b)\in \Big([0,p)\times [1,p)\Big)\cap\ZZ^2
    : 0\le a+s_j<p,\ 
    1\le b+t_j<p
    \Big\}\,,
\end{equation*}
and let $\bx_{a,b}=(x_1,\dots,x_N)$, where
$x_j:=\fq_p\big((a+s_j)p+(b+t_j)\big)/p$ for $1\le j\le N$
and $(a,b)\in\cG_N(p)$.
Then, for any $\bh=(h_1,\dots,h_N)\in\big([-H,H]\cap\ZZ\big)^N$, 
for which $h_J\neq 0$, we have
\begin{equation}\label{eqLemmaExpSum}
   S(p):=\sum_{(a,b)\in\cG_N(p)}e(\langle \bh,\bx_{a,b}\rangle) = O_{M,N}(p)\,.
\end{equation}
\end{lemma}
\begin{proof}
First, let us complete the sum $S(p)$ to a sum over the entire torus $[0,p)^2$.
Let
\begin{equation*}
  \begin{split}
    S_0(p) :=& \sum_{(a,b)\in[0,p)\times[1,p)}e(\langle \bh,\bx_{a,b}\rangle), 
  \end{split}
\end{equation*}
where $\bx_{a,b}=(x_1,\dots,x_N)$ for $(a,b)\not\in\cG_N(p)$ 
are defined by the same formula in the lemma 
except that the arguments are replaced by 
the representatives of $a+s_j$ and $b+t_j$ modulo $p$
taken from $\{0,\dots,p-1\}$.
Then, if $p$ is sufficiently large, the sums $S(p)$ and $S_0(p)$ 
coincide except for at most $4Mp$ new terms added in $S_0(p)$, that is,
\begin{equation}\label{eqLemmaSS0}
  \begin{split}
    S(p) = S_0(p)+O_M(p)\,.  
  \end{split}
\end{equation}
The complete sum is
\begin{equation*}
  \begin{split}
    S_0(p) 
     &= \sum_{a=0}^{p-1}\sum_{b=1}^{p-1}
     e({h_1x_1+\cdots h_Nx_N})\\
      &= \sum_{a=0}^{p-1}\sum_{b=1}^{p-1}
     e\bigg(\frac{h_1\fq_p\big((a+s_1)p+(b+t_1)\big)+\cdots 
     +h_N\fq_p\big((a+s_N)p+(b+t_N)\big)}{p}\bigg)\,.
  \end{split}
\end{equation*}
By Proposition~\ref{PropositionFQ}, we know that for $1\le j\le N$ we have
\begin{equation*}
  \begin{split}
    \fq_p\big((a+s_j)p+(b+t_j)\big) =  \fq_p(b+t_j)-(a+s_j)(b+t_j)^{-1}.
  \end{split}
\end{equation*}
Then, $S_0(p)$ can be written as
\begin{equation}\label{eqLemmaExpSum0}
  \begin{split}
    S_0(p) 
        &= \sum_{b=1}^{p-1}
     e\bigg(\sdfrac{h_1\big(\fq_p(b+t_1)-s_1(b+t_1)^{-1}\big)+\cdots 
     +h_N\big(\fq_p(b+t_N)-s_N(b+t_N)^{-1}\big)}{p}\bigg)\\
      & \phantom{\sum_{b=1}^{p-1}\quad}
      \times  \sum_{a=0}^{p-1}
        e\bigg(\frac{aR(b;\bh,\cV)}{p}\bigg)\,,
  \end{split}
\end{equation}
where $X\mapsto R(X;\bh,\cV)$ is a rational function
in $\FF_p(X)$.
More explicitly,  
\begin{equation*}
  R(X;\bh,\cV)=\frac{P(X,\bh,\cV)}{Q(\bh,b,\cV)},
\end{equation*}
where $P(X;\bh,\cV)$ and $Q(X;\bh,\cV)$ are polynomials over $\FF_p[X]$ of degrees
\begin{equation*}
  0\le \deg(P)<\deg(Q)\quad \text{ and }\quad 1\le \deg(Q)\le N.
\end{equation*}
Note that the roots of $Q(X;\bh,\cV)$ belong to the set
$\{-t_1, \dots, -t_N\}$, but not all elements of the set are always roots, 
as certain terms may cancel each other through addition.
Nevertheless, the requirement in the hypothesis, which
implies that one of the roots is distinct from all the others,
ensures that $\deg(Q) \ge 1$, so that, in particular, 
$R(X;\bh,\cV)$ is not identically zero.

Then, for any $b$ with $1\le b\le p-1$, there are two possibilities:
either $b$ is a root of $P(X;\bh,\cV)$ or  $b$ is not a root of $P(X;\bh,\cV)$.
In the first case, in the sum over $a$ in~\eqref{eqLemmaExpSum0}
all terms are equal to $1$, while in the second case, 
the terms run over the set of all roots of unity of order $p$. Therefore
\begin{equation}\label{eqLemmaExpSumOvera}
  \begin{split}
    \sum_{a=0}^{p-1}
        e\bigg(\frac{aR(b;\bh,\cV)}{p}\bigg)
    = \begin{cases}
        p & \text{ if $R(b;\bh,\cV)\equiv 0\pmod p$,}\\[2mm]
        0 & \text{ if $R(b;\bh,\cV)\not\equiv 0\pmod p$}.
    \end{cases}
  \end{split}
\end{equation}
Then, since $P(X;\bh,\cV)$ has at most $N$ roots modulo $p$, 
on inserting~\eqref{eqLemmaExpSumOvera} in~\eqref{eqLemmaExpSum0},
yields
\begin{equation}\label{eqLemmaSS01}
  \begin{split}
    \big|S_0(p)\big| \le Np\,.  
  \end{split}
\end{equation}
The lemma then follows on combining~\eqref{eqLemmaSS0} and~\eqref{eqLemmaSS01}.
\end{proof}

The next lemma can be deduced from the general literature 
(see Niederreiter~\cite{Nie21978} or Dick and Pillichshammer~\cite{DP2010}), 
but for completeness we will include it along with a brief justification.
\begin{lemma}\label{Lemmarh}
Let $H,N\ge 1$ be integers,
let $\bh=(h_{1},\dots ,h_{N})\in \ZZ^{N}$,
and let $r(\bh)$ be the `appended volume' volume of the associated hypercube defined by
$r(\bh) :=\prod _{j=1}^{N}\max\{1,|h_{j}|\}$.
Then
\begin{equation}\label{eqrbound}
	\sum_{0<\|\bh\|_\infty\le H}\frac{1}{r(\bh)}=O_{N}\big(\log^N H\big).
\end{equation}
\end{lemma}
\begin{proof}

	If we assume exactly $\ell$ entries of $\bh$ are nonzero, then the corresponding sum is bounded by
\begin{equation*}
	\ll \binom{N}{\ell}\sum_{1\le |h_1|\le H}\cdots\sum_{1\le |h_\ell|\le H}\frac{1}{|h_1|\cdots|h_\ell|}\ll \binom{N}{\ell}(\log H)^N.
\end{equation*}
	Hence, as $\ell$ ranges over $1,\ldots,N$, we have
\begin{equation*}
	\sum_{0<\|\bh\|_\infty\le H}\frac{1}{r(\bh)}\ll \sum_{\ell=1}^N\binom{N}{\ell}(\log H)^N < 2^N\log^N H.
\end{equation*}
\end{proof}

\section{Average distances from the the Fermat Quotients to a line}\label{SectionDistanceToALine}

Let $p$ be a large prime number. Let $L$ be a positive integer less than $p$ whose value will be chosen later. Define
\begin{equation*}
 \cT_{jk} = \left(\sdfrac{j-1}{L},\sdfrac{j}{L}\right]
        \times \left(\sdfrac{k-1}{L},\sdfrac{k}{L}\right],\quad 
  1\le j\le L,~2\le k\le L,    
\end{equation*}
while, for $k=1$, we keep the left end of the second interval closed, that is,
\begin{equation*}
 \cT_{j1} = 
    \left(\sdfrac{j-1}{L},\sdfrac{j}{L}\right]
      \times \left[0,\sdfrac{1}{L}\right],\quad 
      1\le j\le L.
\end{equation*}
These rectangles are disjoint and form the set partition
\begin{equation*}
  (0,1]\times[0,1] =\bigcup_{j=1}^L
  \bigcup_{k=1}^L
 \cT_{jk}\,. 
  \end{equation*}
Next, consider the set of normalized points based on the 
Fermat quotient matrix:
\begin{equation*}
  \cB_{jk}:=\Big\{1\le b\le p-1 : \Big(
  \sdfrac{b}{p},\sdfrac{\fq_p(b)}{p}\Big)\in\cT_{jk}\Big\}
  \ \ \text{ for $1\le j,k\le L$}.
\end{equation*}
Note that, since $b/p\ne0$, each $b\in\{1,\ldots,p-1\}$ lies in exactly one of the sets $\cB_{jk}$. 

\medskip

Now, let $C, D> 0$ be fixed real numbers and
let  
\begin{equation*}
   f(x,y):=\big|y-\{Cx+D\}\big|\,.   
\end{equation*}
In order to prove Theorem~\ref{Theorem1}, we need to estimate the average
\begin{equation}\label{eqAverage}
    M(p,C,D):=\frac{1}{p}\sum_{b=1}^{p-1}f\left(\frac{b}{p},\frac{\fq_p(b)}{p}\right).
\end{equation}
With the notations of the partition above, this is
\begin{equation}\label{eqAverageSplit}
    M(p,C,D)
    =\frac{1}{p}\sum_{j=1}^L\sum_{k=1}^L\sum_{b\in \cB_{jk}}f\left(\frac{b}{p},\frac{\fq_p(b)}{p}\right).
\end{equation}

Note that for each $b\in \cB_{jk}$, we have
\begin{equation}\label{eqfO}
f\left(\frac{b}{p},\frac{\fq_p(b)}{p}\right)=f\left(\frac{j}{L},\frac{k}{L}\right)+O\left(\frac{C}{L}\right).   
\end{equation}

Next we estimate the size of the sets $\cB_{jk}$. 
Note that $b\in \cB_{jk}$ if and only if both of 
the following conditions are satisfied:
\begin{enumerate}
    \item[(i)]
    $\sdfrac{j-1}{L}p<b\le\sdfrac{j}{L}p$,\\[1mm]
    \item[(ii)]
    $ \sdfrac{k-1}{L}p<\sdfrac{\fq_p(b)}{p}\le\sdfrac{k}{L}p$
    \ \  for $1\le k\le L$.
\end{enumerate}

The size of $\cB_{jk}$ will follow if we know the discrepancy of the sequence
\begin{equation*}
    \cS_p(j)=\left\{\frac{\fq_p(b)}{p} :  \sdfrac{j-1}{L}p<b\le\sdfrac{j}{L}p\right\}\subset [0,1]\ \ 
    \text{ for $1\le j\le L$.}
\end{equation*}
According to Erd\H os-Tur\'an inequality~\cite[Corollary 1.1]{Mon1971},
the discrepancy is bounded by
\begin{equation*}
    D\big(\cS_p(j)\big) \le
    \frac{p/L}{K+1}+3\sum_{m=1}^K\frac{1}{m}
    \bigg|\sum_{\frac{j-1}{L}p<b\le\frac{j}{L}p}e\left(\frac{m\fq_p(b)}{p}\right)\bigg|.
\end{equation*}
The exponential sum above is bounded by Heath-Brown's Lemma~\ref{LemmaHB}:
\begin{equation*}
    \sum_{\frac{j-1}{L}p<b\le\frac{j}{L}p}e\left(\frac{m\fq_p(b)}{p}\right)
    \ll 
    \left(\frac{p}{L}\right)^{1/2}p^{3/8}=L^{-1/2}p^{7/8},
\end{equation*}
for all integers $m$ relatively prime to $p$.
Therefore, it follows that
\begin{equation*}
  \begin{split}
    D\big(\cS_p(j)\big) &\ll
    \sdfrac{p}{L(K+1)}+3\sum_{m=1}^K\sdfrac 1m L^{-1/2}p^{7/8}\\
    &\le \frac{1}{\sqrt{L}}\left(\frac{p}{K+1}+3p^{7/8}\log K\right),
  \end{split}
\end{equation*}
for any integer $K\ge 1$. To balance the terms,
we choose $K=\lfloor p^{1/8}\rfloor$, and obtain
\begin{equation*}
  \begin{split}
    D\big(\cS_p(j)\big) &\ll L^{-1/2} p^{7/8}\log p\,.
  \end{split}
\end{equation*}
 Thus, by the definition of the discrepancy,
\begin{equation}\label{eqBjk}
  |\cB_{jk}|=\frac{p}{L}\cdot\frac{1}{L}+O\left(L^{-1/2}p^{7/8}\log p\right),    
\end{equation}
for $1\le j,k\le L$.

Now we can estimate the average defined by~\eqref{eqAverage}.
For this, we use the decomposition in rectangles~\eqref{eqAverageSplit},
where the values of $f$ can be approximated by those on the 
corners~\eqref{eqfO}:
\begin{equation*}
  \begin{split}
     M(p,C,D) &=\frac{1}{p}\sum_{j=1}^L\sum_{k=1}^L\sum_{b\in \cB_{jk}}f\left(\frac{b}{p},\frac{\fq_p(b)}{p}\right)\\
     &=\frac{1}{p}\sum_{j=1}^L\sum_{k=1}^L
     |\cB_{jk}|
     \bigg(f\Big(\sdfrac{j}{L},\sdfrac{k}{L}\Big)+O\Big(\sdfrac CL\Big)\bigg)\,.
  \end{split}
\end{equation*}
On inserting the size of $|\cB_{jk}|$ from~\eqref{eqBjk}, yields
\begin{equation}\label{eqMmed}
  \begin{split}
   M(p,C,D) 
   =&\frac{1}{p}\left(\frac{p}{L^2}+O\left(L^{-1/2}p^{7/8}\log p\right)\right)
   \sum_{j=1}^L\sum_{k=1}^L
   \left(f\left(\frac{j}{L},\frac{k}{L}\right)+O\left(\frac{C}{L}\right)\right)
	\\
     =&\frac{1}{L^2}\sum_{j=1}^L\sum_{k=1}^L
 f\left(\frac{j}{L},\frac{k}{L}\right)
 +O\bigg(\frac{\log p}{L^{1/2}p^{1/8}}\sum_{j=1}^L\sum_{k=1}^L
 f\left(\frac{j}{L},\frac{k}{L}\right)\bigg)
	\\
     &+O\bigg(\sdfrac{C}{L^3}\sum_{j=1}^L\sum_{k=1}^L 1\bigg)
      +O\left(CL^{1/2}p^{-1/8}\log p\right).
  \end{split}
\end{equation}
Next, to continue, it should be noted that, by Lemma~\ref{LemmaIntegral2}, the first term
on the right-hand side of~\eqref{eqMmed},
which is a Riemann sum, tends to $1/3$ if $L\to\infty$,
because $C\neq 0$. 
Furthermore, counting on the maximal possible offset makeing use of~\eqref{eqfO}, we find that
\begin{equation}\label{eqIntR}
   \frac{1}{L^2}\sum_{j=1}^L\sum_{k=1}^L
    f\left(\frac{j}{L},\frac{k}{L}\right)
    =\frac{1}{3}+O\left(\sdfrac{C}{L}\right).
\end{equation}
Using~\eqref{eqIntR} again, it follows that the sum of the three
error terms in~\eqref{eqMmed} is
\begin{equation}\label{eqErr}
   O\left(\sdfrac{C}{L}+(L+C)L^{1/2}p^{-1/8}\log p\right)
   =O\left(\sdfrac{C}{L}+L^{3/2}p^{-1/8}\log p\right),
\end{equation}
because we need to choose $L$ such that $C=o(L)$,
in order to have a proper estimate with the error term that is not 
larger than the main one, which is equal to $1/3$, as we have 
learned from~\eqref{eqIntR}.

Now we choose $L=C^{2/5}p^{1/20}\log^{-2/5}p$, 
so that the terms in~\eqref{eqErr} have equal contributions.
Then, under the assumption that $C=o\big(p^{1/12}\log^{-2/3} p\big)$,
from~\eqref{eqMmed},\eqref{eqIntR}, and~\eqref{eqErr}
it follows
\begin{equation*}
  M(p,C,D) = \frac{1}{3}+O\big(C^{3/5}p^{-1/20}\log^{2/5}p\big)\,,  
\end{equation*}
with an absolute constant in the error term.
This concludes the proof of Theorem~\ref{Theorem1}.

\section{Distribution of Fermat quotients in geometric configuration respecting relative size conditions. Proof of Theorem~\ref{Theorem2}}\label{SectionPatterns}

The proof of Theorem~\ref{Theorem2} goes through the following steps.
First, we decompose the unit cube that contain the $N$-dimensional
vectors $\bx_{a,b}$ into boxes of a generic size and
then we estimate the number of vectors $\bx_{a,b}$ in each box.
Next, we count the number of boxes that contain ``good'' vectors,
meaning those that follow the increasing order given by $\sigma$.
Finally, we add up all the numbers of good vectors and choose the 
optimal size value of the boxes for which the error term is smaller.

\subsection{Breaking the unit cube into pieces}\label{SubsectionBoxes}

Let $L>1$, whose precise value will be chosen later, be fixed.
We divide the unit cube into $L^N$ boxes:
\begin{equation*}
   \scB_{i_1,\ldots,i_N} = 
 \Big[\sdfrac{i_1}{L},\sdfrac{i_1+1}{L}\Big)
    \times\cdots\times
 \Big[\frac{i_N}{L},\frac{i_N+1}{L}\Big),\quad 0\le i_1,\ldots,i_N\le L-1.
\end{equation*}

Let $\cV\in(\ZZ^2)^N$ be the ordered set of vectors 
$\bv_1,\dots,\bv_N$ that defines the pattern 
of displacement vectors and let $\cG_N(p)$ be the set of admissible positions in 
the Fermat quotient matrix $\FQMp$ defined by~\eqref{eqdefG}, positions that 
still remain in the matrix after translations by any $\bv\in\cV$.

To evaluate the spread of vectors 
$\bx_{a,b}=\frac{1}{p}(A_{(a,b)+\bv_1},\dots,A_{(a,b)+\bv_N})$ with
$A_{a,b}=\fq_p(ap+b)$ and
$(a,b)\in\cG_N(p)$, we will use the fact that their set is uniformly distributed in $[0,1]^N$. 
In order to verify this fact, let us estimate the discrepancy of
\begin{equation*}
	\cX=\{\bx_{a,b}:(a,b)\in\cG_N(p)\}.
\end{equation*}

The Koksma--Sz\"usz inequality~\eqref{LemmaETK} imply that for any
integer $H>1$, we have
\begin{equation*}
\frac{\#\big(\cX\cap \scB_{i_1,\dots,i_N}\big)} {|\cX|}
  = \mu(\scB_{i_1,\dots,i_N})
  +O_N    \bigg(\frac{1}{H}
    +\sum _{0<\|\bh\|_{{\infty}\leq H}}\frac  {1}{r(\bh)}
    \bigg|{\frac{1}{|\cX|}}\sum_{(a,b)\in\cG_N(p)}
    e\Big(2\pi i  \langle \bh,\bx_{j}\rangle\Big)\bigg|\bigg),    
\end{equation*}
On using~\eqref{eqCardGN} and the bound~\eqref{eqLemmaExpSum} for the exponential sum, it follows
\begin{equation*}
\#\big(\cX\cap \scB_{i_1,\dots,i_N}\big) 
  =\frac{p^2}{L^N}
  +O_{M,N}    \bigg(\frac{p^2}{H}+\frac{p}{L^N}
    +p\sum _{0<\|\bh\|_{{\infty}\leq H}}\frac  {1}{r(\bh)}
    \bigg)\,.   
\end{equation*}
It should be noted that we were able to apply~\eqref{eqLemmaExpSum}
because the condition $0<\|\bh\|_{{\infty}}$ in the summation
ensures that there exists $j$ for which $h_j\neq 0$, 
call $J$ this index, and from the hypothesis of the theorem we know that
$t_J$ is different from each other $t_j$ with $1\le j\le N$ and~$j\neq J$.

Bounding the sum over $\bh$ by~\eqref{eqrbound} and taking $H=p$, we conclude
\begin{equation}\label{eqcardB}
 \#\big(\cX\cap \scB_{i_1,\dots,i_N}\big) 
 =\frac{p^2}{L^N}
  +O_{M,N}    \big(p\log^N p \big)\,.
\end{equation}

\subsection{Counting on the boxes that contain ordered vectors}\label{SubsectionTetrahedron}
Let $\sigma$ be a permutation of $N$ elements and let
$T(\sigma,N)$ be the polyhedron
\begin{equation*}
    T(\sigma,N)=\big\{(x_1,\ldots,x_N)\in [0,1]^N : 
    x_{\sigma(1)}<x_{\sigma(2)}<\cdots<x_{\sigma(N)}\big\},
\end{equation*}
which we know from Lemma~\ref{LemmaTetrahedron} that has volume 
$\mu\big(T(\sigma,N)\big)=1/N!$.

Note that some of these boxes are completely included in $T(\sigma,N)$, 
others are only partially, intersecting the boundary of $T(\sigma,N)$, 
while others are outside of $T(\sigma,N)$.
Precisely, if
$i_{\sigma(1)}<\cdots<i_{\sigma(N)}$, then 
$\scB_{i_1,\ldots,i_N}\subseteq T(\sigma,N)$. If any of the inequalities becomes an equality, then points in $\scB_{i_1,\ldots,i_N}$ may or may not lie in $T(\sigma,N)$. If any of the inequalities gets reversed, then points in $\scB_{i_1,\ldots,i_N}$ 
do not lie in $T(\sigma,N)$.

We need to know the number of boxes included in  $T(\sigma,N)$, 
respectively the ones that have a non-empty intersection with  $T(\sigma,N)$.

\begin{remark}
{\upshape 1.}  Let $N,L$  be integers such that $1\le N\le L$. Then
\begin{equation}\label{eqRemark1}
    B^{*}(N,L):=\#\big\{(i_1,\dots,i_N)\in\{1,\dots,L\}^N :
    i_1<\cdots < i_N\big\}
    =\binom{L}{N}.
\end{equation}
This follows by an induction argument by counting the numbers of tuples with $i_N$ on all possible positions and observing that
\begin{equation*}
    \binom{L-1}{N-1}+\binom{L-2}{N-1}+\cdots+\binom{N-1}{N-1}
    =\binom{L}{N},
\end{equation*}
equality known as the ``hockey stick formula''.
\smallskip

{\upshape 2.}  Let $N,L$  be integers such that $1\le N\le L$. Then
\begin{equation}\label{eqRemark2}
    B(N,L):=\#\big\{(i_1,\dots,i_N)\in\{1,\dots,L\}^N :
    i_1\le \cdots \le i_N\big\}
    =\binom{L+N-1}{N}.
\end{equation}
The exact equality above follows by induction
in the same way as~\eqref{eqRemark1}, 
the only difference in the counting being the allowance for $i_N$ 
to be equal to $i_{N-1}$.
\end{remark}
Then,~\eqref{eqRemark1} and~\eqref{eqRemark2} imply that the number of
boxes $\scB_{i_1,\dots,i_N}$ included in $T(\sigma,N)$ and the number of boxes
that intersect $T(\sigma,N)$ are approximately equal as $L$ becomes sufficiently large:
\begin{equation}\label{eqBBstar}
   B^{*}(N,L) = \frac{L^N}{N!} +O_N\big(L^{N-1}\big)\ \ \text{ and }\ \  
   B(N,L) = \frac{L^N}{N!} +O_N\big(L^{N-1}\big).
\end{equation}

\subsection{Completion of the proof of Theorem~\ref{Theorem2}}\label{SubsectionProofTheorem2}
Knowing, on the one hand, the cardinality of `good' vectors 
$\bx_{a,b}$ in a box $\scB_{i_1,\dots,i_N}$, 
which is given by~\eqref{eqcardB}, and, on the other hand, the number of boxes
that intersect the polyhedron $T(\sigma,N)$, we find that
\begin{equation*}
  \begin{split}
      \#\left\{(a,b)\in\cG:\bx_{a,b}\in T(\sigma,N)\right\}
   &= \#\big(\cX\cap T(\sigma,N)\big)\\
    &= \left(\frac{L^N}{N!}+O_N\big(L^{N-1}\big)\right)
    \left(\frac{p^2}{L^N} +O_{M,N}\big(p\log^N p \big)\right)\\
    & = \frac{p^2}{N!} +O_{M,N}\left(pL^N\log^N p +\frac{p^2}{L}\right).
  \end{split}
\end{equation*}
To balance the error terms, we choose
$L=\lfloor p^{1/(N+1)}\log^{-N/(N+1)p}\rfloor$ and obtain
\begin{equation*}
   \#\left\{(a,b)\in\cG:\bx_{a,b}\in T(\sigma,N)\right\}
    =\frac{p^2}{N!}
    +O_{M,N}\left(p^{(2N+1)/(N+1)}\log^{N/(N+1)}p\right).
\end{equation*}
This concludes the proof of Theorem~\ref{Theorem2}.
\hfill\mbox{\qed}

\vspace{7mm}
\section*{Addenda}
We include here several representations, 
on the one hand of the set of pairs 
$(a,b)\in \cG_N(p)$ for which the vectors $\bx_{a,b}$ defined by~\eqref{eqdefx} belong to the polyhedra $T(\sigma,N)$, and on the other hand, of the sets of vectors $\cX(p)\cap T(\sigma,N)$ defined by~\eqref{eqX} and~\eqref{eqThtrahedron}.

Let $\cD(\sigma,N,p)$ denote the set of points whose cardinality is estimated in Theorem~\ref{Theorem2}, that~is,
\begin{equation}\label{eqD}
    \cD(\sigma,N,p) = \{(a,b)\in\cG_N(p):
       \bx_{a,b}(p)\in T(\sigma,N)\big\}.
\end{equation}

In the three sets $\cD(\sigma,N,p)$ shown in Figure~\ref{FigureXT1},
one can see that an increasing number of equal second components in displacement vectors $\cV$ causes the distribution of points to change from pseudo-random to increasingly regular.
\begin{table}[ht]
\centering
\caption{The parameters used for generating the set of pairs $(a,b)\in\cD(\sigma,N,p)$ 
in  Figure~\ref{FigureXT1} positioned from left to right 
and numbered sequentially from $1$ to $3$.
The prime is $p=601$
in all cases and the dimension is $N=4$ in the 1st image and $N=3$ in the last two images. Then
$p^2=361201$, $p^2/4! \approx 15050.041$
and $p^2/3! \approx 60200.166$.}
\begin{tabular}{c@{\hskip 1.15em}cccccc}
\toprule
 & $\sigma$ & $\cV$ & $\#\cG_N(p)$ & $\#(\cX(p)\cap T(\sigma,N))$ & 
$\frac{p^2/N}{\#(\cX(p)\cap T(\sigma,3))}$\\
\midrule
{\bf\small 1}& $(2, 4, 1, 3)$ & $(0, 4), (7, 5), (3, 10), (7, 16)$ & $346896$ & $14360$ &  $1.048$ \\
{\bf\small 2}& $(2,3,1)$ & $(10, 11), (1, 6), (2, 6)$ & $348099$ & $57802$ &  $1.041$ \\
{\bf\small 3}& $(2,3,1)$ & $(38, 6), (1, 6), (2, 6)$ & $334422$ & $56512$ &  $1.065$ \\
\bottomrule
\end{tabular}
\label{TableA11}
\end{table}
\begin{figure}[t]
 \centering
 \hfill
    \includegraphics[width=0.32\textwidth]{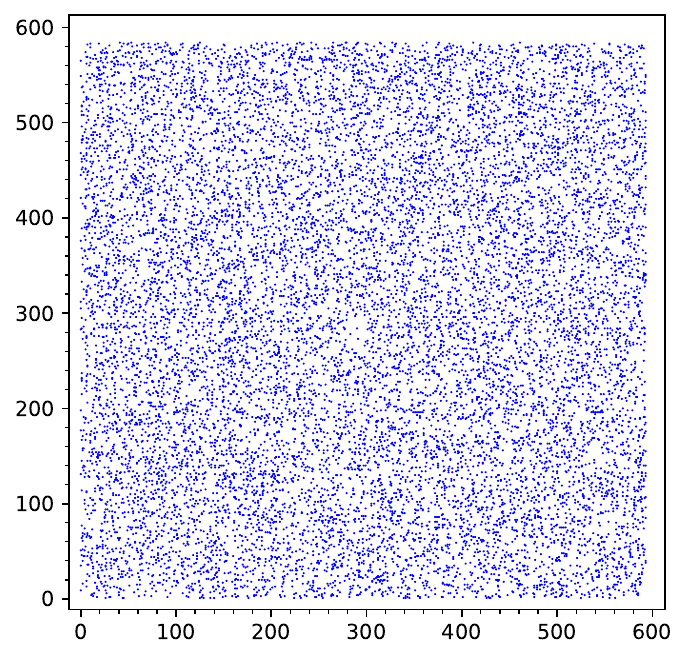}
    \includegraphics[width=0.32\textwidth]{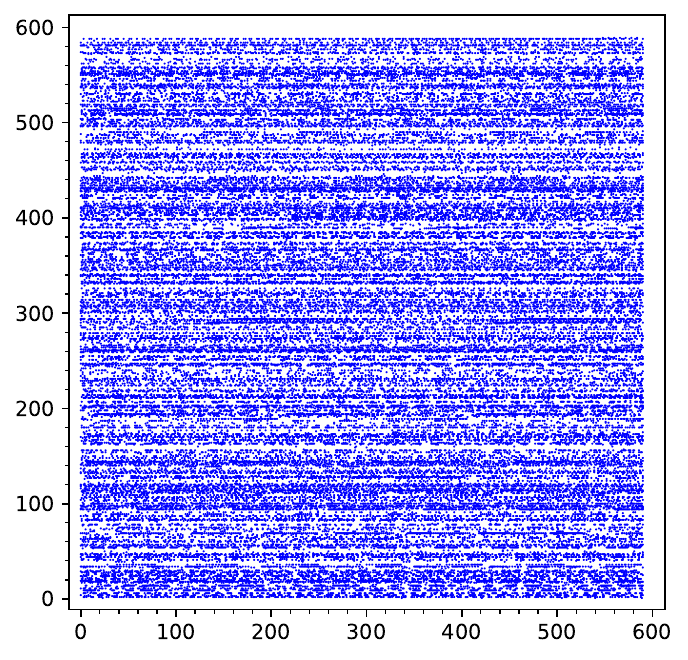}
    \includegraphics[width=0.32\textwidth]{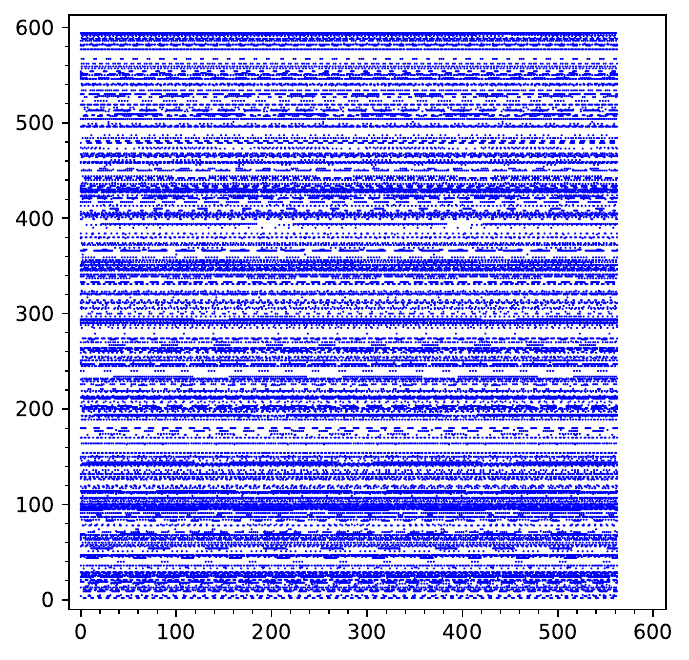}    
    \hfill\mbox{}
\caption{Three representations of the set $\cD(\sigma,N,p)$ in which the
second components of displacement vectors become increasingly similar. 
The exact values of the generating parameters are listed in 
Table~\ref{TableA11}.
} 
 \label{FigureXT1}
 \end{figure}

We also remark that the ratios $p^2/\#(\cX(p)\cap T(\sigma,3))$
might strongly depend on $\sigma$, unlike 
the case treated in the Theorem~\ref{Theorem2}, 
where all the second components of the vectors  in $\cV$ are distinct.
To compare the dependency, see the different proportions in the case of the three images in Figure~\ref{FigureXT2}, where only the changed permutations $\sigma$ make them distinct.

\begin{figure}[htb]
 \centering
 \hfill
    \includegraphics[width=0.32\textwidth]{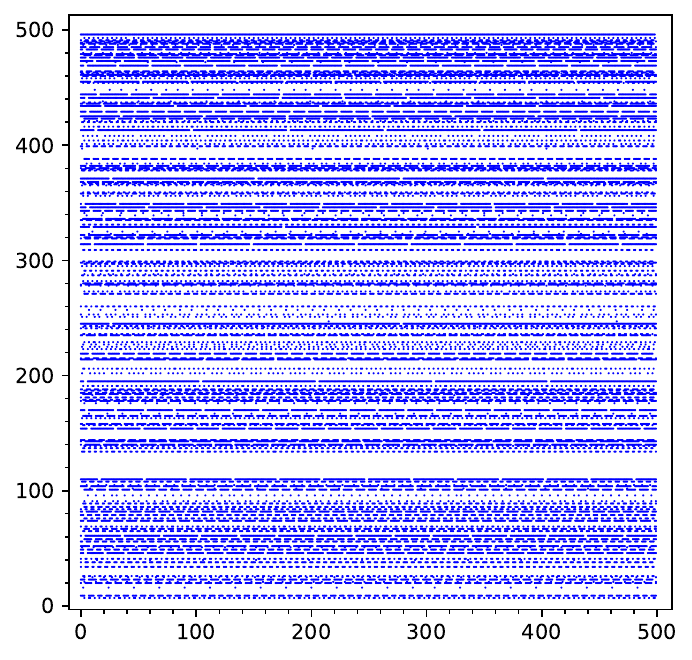}
    \includegraphics[width=0.32\textwidth]{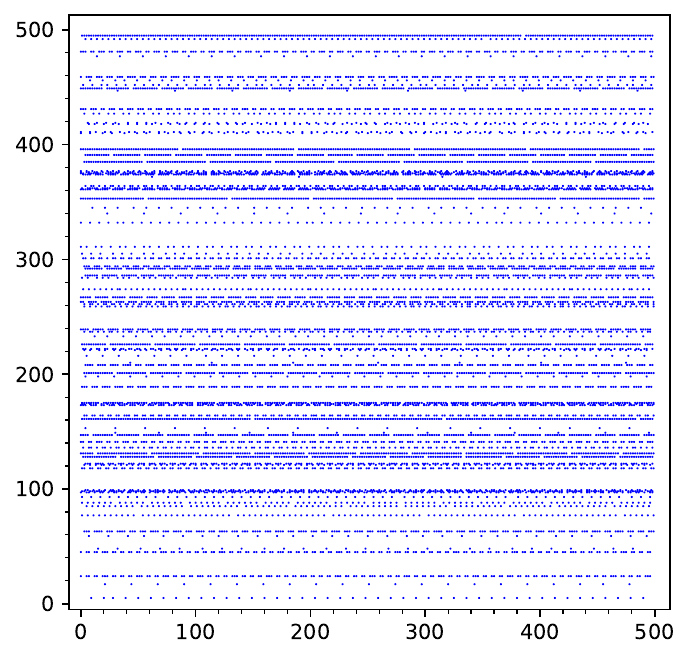}
    \includegraphics[width=0.32\textwidth]{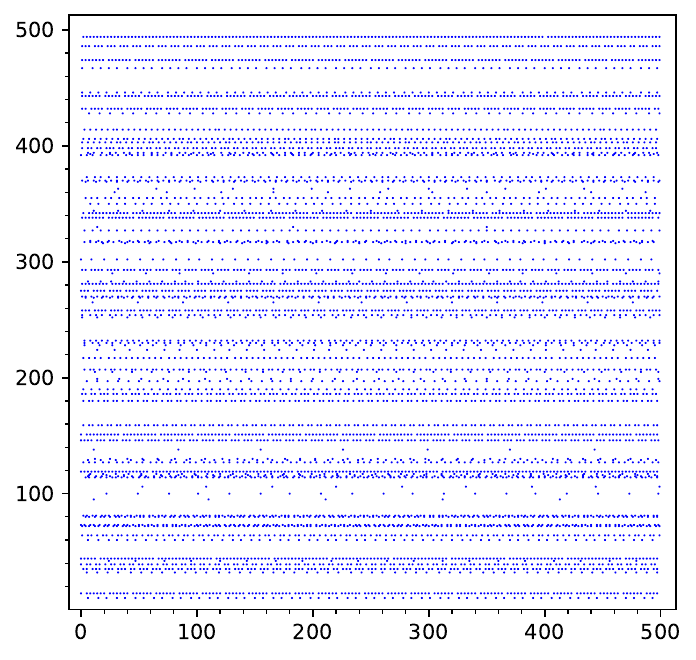}    
    \hfill\mbox{}
\caption{The set of points $\cD(\sigma,N,p)$ with generating parameters shown
in Table~\ref{TableA12}, where only the permutation $\sigma$ is changed from set to set.
} 
 \label{FigureXT2}
 \end{figure}
\begin{figure}[htb]
 \centering
 \hfill
    \includegraphics[width=0.98\textwidth]{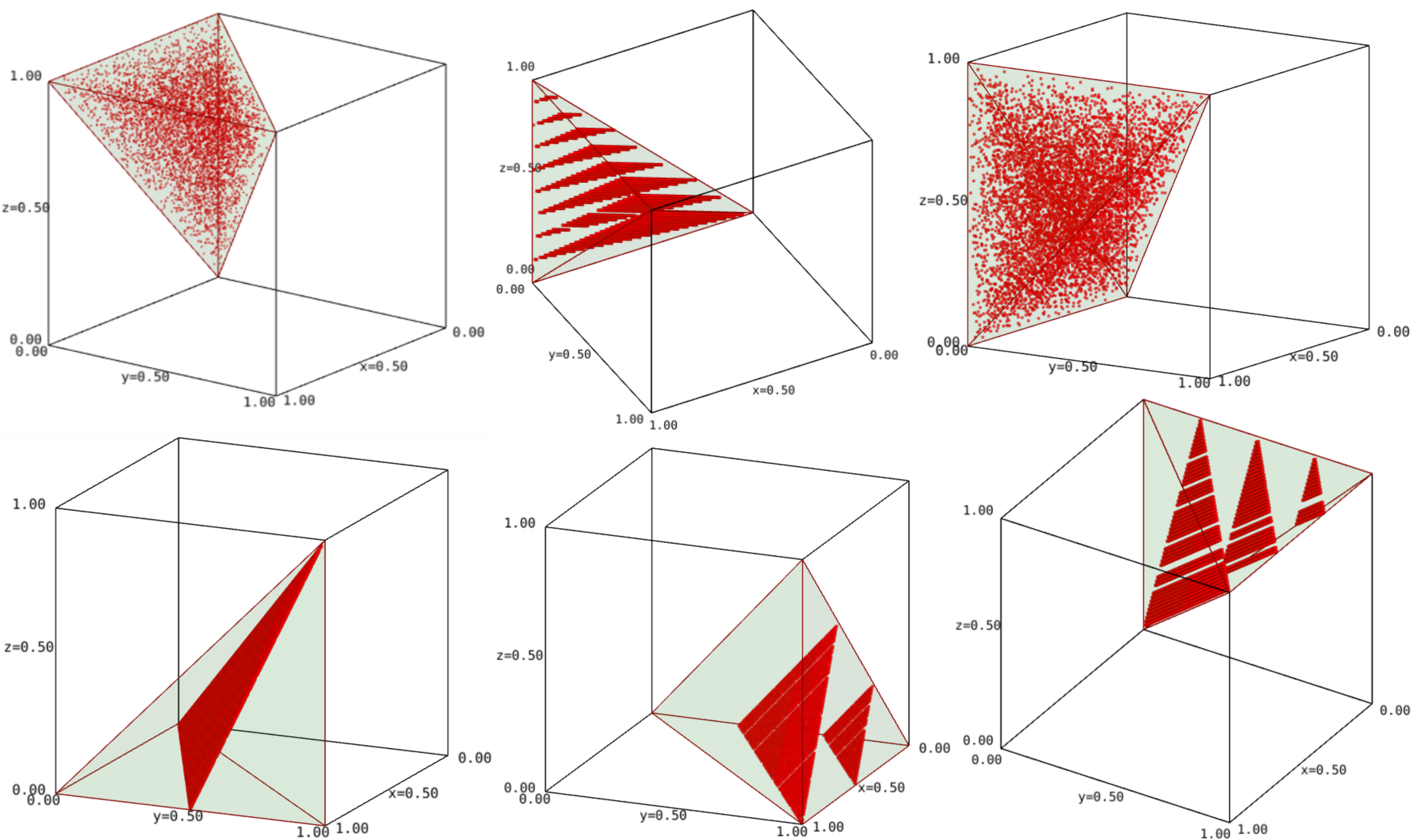}\ \ 
    \hfill\mbox{}
\caption{Six instances representing the set of points $\cX(p)$ belonging to the tetrahedrons $T(\sigma,3)$, with the generating parameters presented in Table~\ref{TableA2}.
} 
 \label{FigurePolyhedrons3d}
 \end{figure}

In Figure~\ref{FigurePolyhedrons3d} are shown, in three dimensions, 
the set of points $\cX(p)$ defined by~\eqref{eqX} that are contained in the
tetrahedrons $T(\sigma,3)$ for different permutations $\sigma$ and
sets of displacement vectors $\cV$.
We remark that if all second components of the displacement vectors are equal, then
the points in $\cX(p)\cap T(\sigma,3)$ 
are no longer equidistributed in the cube $[0,1]^3$, 
but they arrange themselves in several triangles that transversely intersect 
their minimal containing tetrahedron.
It should be noted that the ratio 
\mbox{$p^2/\#\big(\cX(p)\cap T(\sigma,3)\big)$}
may slowly approach a certain value (in Table~\ref{TableA2}, 
the comparison is made with $3!$), 
but the number of parallel transverse triangles, on which the points $\bx_{a,b}$ with
$(a,b)\in \cD(\sigma,3,p)$ are aligned, becomes apparent even for small values of~$p$.

\begin{table}[ht]
\centering
\caption{The parameters used for generating the set of pairs $(a,b)\in\cD(\sigma,N,p)$ 
in  Figure~\ref{FigureXT2} positioned from left to right 
and numbered sequentially from $1$ to $3$.
The prime is $p=503$
in all cases and the dimension is $N=4$. Then
$p^2=253009$ and $p^2/4! \approx 10542.041$.}
\begin{tabular}{c@{\hskip 1.15em}cccccc}
\toprule
 & $\sigma$ & $\cV$ & $\#\cG_N(p)$ & $\#(\cX(p)\cap T(\sigma,N))$ & 
$\frac{p^2/N}{\#(\cX(p)\cap T(\sigma,3))}$\\
\midrule
{\bf\small 1}& $(1, 2, 3, 4)$ & $(0, 6), (1, 6), (2, 6), (3, 6)$ & $248000$ & $41624$ &  $0.253$ \\
{\bf\small 2}& $(1, 3, 2, 4)$ & $(0, 6), (1, 6), (2, 6), (3, 6)$ & $248000$ & $10480$ &  $1.005$ \\
{\bf\small 3}& $(1, 4, 2, 3)$ & $(0, 6), (1, 6), (2, 6), (3, 6)$ & $248000$ & $6912$ &  $1.525$ \\
\bottomrule
\end{tabular}
\label{TableA12}
\end{table}

\begin{table}[ht]
\centering
\caption{The parameters used for generating the set of points $\cX(p)\cap T(\sigma,N)$ 
in  Figure~\ref{FigurePolyhedrons3d} positioned from left to right and top to bottom, 
and numbered sequentially from $1$ to $6$.
In all cases
$p=211$, $N=3$, so that $p^2=44521$ and $p^2/N! \approx 7420.166$.}
\begin{tabular}{c@{\hskip 1.15em}cccccc}
\toprule
 &  $\sigma$ & $\bv_1,\bv_2,\bv_3$ & $\#\cG_N(p)$ & $\#(\cX(p)\cap T(\sigma,N))$ & 
$\frac{p^2/N}{\#(\cX(p)\cap T(\sigma,3))}$\\
\midrule
{\bf\small 1} & $(2,1,3)$ & $(26,-11), (26,12), (26,33)$ & $30710$ & $4995$ &  $1.485$ \\
{\bf\small 2} & $(2,3,1)$ & $(10,6), (1,6), (2,6)$ & $41004$ & $7256$ &  $1.022$ \\
{\bf\small 3} & $(2,3,1)$ & $(10,11), (1,6), (2,6)$ & $39999$ & $6555$ &  $1.131$ \\
{\bf\small 4} & $(3,2,1)$ & $(1,16), (4,16), (7,16)$ & $39576$ & $10070$ &  $0.736$ \\
{\bf\small 5} & $(3,1,2)$ & $(3,12), (6,12), (12,12)$ & $39402$ & $5338$ &  $1.390$ \\
{\bf\small 6} & $(1,2,3)$ & $(91,26), (2,26), (5,26)$ & $37904$ & $7125$ &  $1.041$ \\

\bottomrule
\end{tabular}
\label{TableA2}
\end{table}


\vspace*{1mm}
\subsection*{Acknowledgement}
We are very grateful to the anonymous reviewers for their efforts in reading carefully our manuscript 
and for their valuable comments and suggestions.
\vspace{5mm}

\vspace{6mm}
\vspace*{2mm}



\vspace*{10mm}

\end{document}